\documentclass[12pt]{amsart}
\usepackage{amscd}
\usepackage{amssymb}
\usepackage{amsmath}
\usepackage{amsfonts}

\theoremstyle{plain}
    \newtheorem{theorem}{Theorem}[section]
    \newtheorem{corollary}[theorem]{Corollary}
    \newtheorem{lemma}[theorem]{Lemma}
    \newtheorem{proposition}[theorem]{Proposition}
    \newtheorem*{theorem*}{Theorem}
    \newtheorem*{proposition*}{Proposition}
    \newtheorem*{corollary*}{Corollary}
    \newtheorem*{theorem335*}{Theorem {\rm (\reft{fi})}}
    \newtheorem*{theorem6410*}{Theorem {\rm (\reft{lg3})}}

\theoremstyle{definition}
    \newtheorem*{definition}{Definition}

    \newtheorem*{obs*}{Notation}
    
    \newtheorem{remark}[theorem]{Remark}
    \newtheorem*{remarks*}{Remarks}
    \newtheorem*{remark*}{Remark}
    \newtheorem{example}[theorem]{Example}

    \newtheorem{clas*}{Exercises}
    
    \newtheorem*{claimno*}{Examples}
    \newtheorem*{claim*}{Example}

    \newcommand{\reft}[1]{Theo\-rem~\ref{#1}}
    \newcommand{\refp}[1]{Pro\-po\-si\-tion~\ref{#1}}
    \newcommand{\refl}[1]{Lem\-ma~\ref{#1}}
    \newcommand{\refc}[1]{Co\-rol\-lary~\ref{#1}}

    \newcommand{\refr}[1]{Re\-mark~\ref{#1}}
    
    \newcommand{\refts}[1]{Theo\-rems~\ref{#1}}
    \newcommand{\refps}[1]{Pro\-po\-si\-tions~\ref{#1}}

\newcommand{\N}{\mathbb N}

\newcommand{\E}{\mathbb E}

\newcommand{\bI}{\mathbf{I}}

\newcommand\alp{\alpha}

\newcommand{\nor}{Noe\-the\-rian ring}

\newcommand{\nol}{Noe\-the\-rian lo\-cal ring}

\newcommand{\rmo}{$R$-mo\-du\-le}

\newcommand{\fgr}{fi\-ni\-te\-ly ge\-ne\-ra\-ted $R$-mo\-du\-le}

\newcommand{\fg}{fi\-ni\-te\-ly ge\-ne\-ra\-ted}
\newcommand{\cm}{Cohen-Macau\-lay lo\-cal ring}

\newcommand{\CM}{Cohen-Macau\-lay}

\newcommand{\eq}{equi\-mul\-ti\-ple}
\newcommand{\ci}{com\-ple\-te in\-ter\-sec\-tion}
\newcommand{\gci}{ge\-ne\-ri\-cally a com\-ple\-te in\-ter\-sec\-tion}

\newcommand{\tf}{tor\-sion\-free }
\newcommand{\fgtfrmo}{fi\-ni\-te\-ly ge\-ne\-ra\-ted tor\-sion\-free
$R$-mo\-du\-le }

\newcommand{\flps}{free lo\-cally on the punc\-tu\-red spec\-trum}

\newcommand{\seq}[2]{({#1}_1,\dotsc,{#1}_{#2})}

\newcommand{\sek}[2]{{#1}_1,\dotsc,{#1}_{#2}}

\newcommand{\ses}[2]{{#1}_1 \+ \cdots \+ {#1}_{#2}}

\newcommand{\bct}{\begin{center} \begin{tabular}}
\newcommand{\ect}{\end{tabular} \end{center}}

\newcommand{\prd}{^{\prime\prime}}

\newcommand{\mat}{\begin{bmatrix}}
\newcommand{\emat}{\end{bmatrix}}
\newcommand{\edet}{\end{matrix} \right|}

\newcommand{\bca}{\begin{cases}}
\newcommand{\eca}{\end{cases}}

\newcommand{\lpar}{\left(}
\newcommand{\rpar}{\right)}

\renewcommand{\det}{\left| \begin{matrix}}

\newcommand{\col}{\colon}

\newcommand{\apl}[3]{#1 \col #2 \ra #3}

\newcommand{\x}{\times}
\newcommand{\ox}{\otimes}
\newcommand{\oxr}{\otimes_{R}}
\newcommand{\+}{\oplus}
\newcommand{\bop}{\bigoplus}

\newcommand{\sube}{\subseteq}
\newcommand{\subn}{\subsetneq}
\newcommand{\nsub}{\nsubseteq}

\newcommand{\sub}{\subset}

\newcommand{\sse}{if and only if\ }

\newcommand{\ra}{\rightarrow}

\newcommand{\LRa}{\Leftrightarrow}

\newcommand{\hra}{\hookrightarrow}
\newcommand{\tra}{\twoheadrightarrow}

\newcommand{\dis}{\displaystyle}

\newcommand{\cS}{{\mathcal S}}

\newcommand{\CR}{{\mathcal R}}

\newcommand{\cF}{{\mathcal F}}

\newcommand{\fm}{{\mathfrak{m}}}

\newcommand{\frp}{{\mathfrak{p}}}
\newcommand{\frq}{{\mathfrak{q}}}

\newcommand{\frk}{k}

\DeclareMathOperator{\rank}{rank}

\DeclareMathOperator{\spec}{Spec}

\DeclareMathOperator{\Ht}{ht}

\DeclareMathOperator{\Min}{Min}
\DeclareMathOperator{\supp}{Supp}
\DeclareMathOperator{\Ass}{Ass}
\DeclareMathOperator{\ann}{ann}

\DeclareMathOperator{\Quot}{Quot}

\DeclareMathOperator{\Hom}{Hom}
\DeclareMathOperator{\Ext}{Ext}

\DeclareMathOperator{\z}{Z}

\DeclareMathOperator{\proj}{proj}
\DeclareMathOperator{\depth}{depth}
\DeclareMathOperator{\grade}{grade}

\DeclareMathOperator{\ad}{ad}

\DeclareMathOperator{\de}{d}

\newcommand{\sR}{\spec(R)}

\newcommand{\Rm}{(R,\fm)}
\newcommand{\Rmk}{(R,\fm, \frk)}

\newcommand{\pjd}{\proj \dim \,}
\newcommand{\dpt}{\depth \,}
\newcommand{\grd}{\grade \,}

\newcommand{\lE}{\ell(E)}
\newcommand{\RE}{\CR(E)}
\newcommand{\FE}{\cF(E)}
\newcommand{\FV}{\cF(V)}
\newcommand{\SE}{\cS(E)}
\newcommand{\lI}{\ell(I)}
\newcommand{\RI}{\CR(I)}
\newcommand{\FI}{\cF(I)}

\newcommand{\ol}{\overline}

\newcommand{\biind}[3]{\renewcommand{\arraystretch}{0.5} #1_{%
\hspace{-1mm}\begin{array}[t]{l} {\scriptstyle #2}\\ {\scriptstyle #3}
\end{array}} \renewcommand{\arraystretch}{1}}

\allowdisplaybreaks

\begin{document}

\title[On equimultiple modules]
{On equimultiple modules}
\author{Ana L. Branco Correia}
\address{Centro de Estruturas Lineares e Combinat\'orias
\\ Universidade de Lisboa
\\ Av. Prof. Gama Pinto 2  \\ 1649-003 Lisboa
\\ Portugal}
\email{alcorreia@cii.fc.ul.pt}
\address{ Universidade Aberta \\
Rua Fern\~ao Lopes $2^{\text{o}}$ Dto
\\ 1000-132 Lisboa \\ Portugal
}
\email{matalrbc@univ-ab.pt}
\thanks{}
\thanks{}
\author{Santiago Zarzuela}
\address{Departament d'\`Algebra i Geometria \\ Universitat de
Barcelona \\ Gran Via 585 \\E-08007 Barcelona \\ Spain}
\email{szarzuela@ub.edu}
\thanks{The second author has been partially supported by MTM2004-01850 (Spain)}
\thanks{}
 \keywords{}
\date{}
\dedicatory{}
\commby{}

\begin{abstract}
    We study the class of \eq{} modules.  In particular, we prove several criteria for
    an \eq{} module to be a \ci{} and prove the openness of the \eq{} locus of an ideal module.
\end{abstract}

\maketitle

\section{Introduction}

\noindent Equimultiple ideals (i.e., analytic spread = height)
have been extensively studied partly because of their connections
to geometry. This notion is the algebraic formulation of the
concept of \eq{} variety introduced by O. Zariski, which is of
great importance in several aspects within the study of algebraic
singularities. We refer to the article of J. Lipman \cite{jl} or
the book \cite{hio} by M. Herrmann, S. Ikeda and U. Orbanz for a
detailed explanation of these connections.  On the other hand, the
re\-le\-van\-ce of equimultiple ideals is also focused on a
theorem by E. B\"oger (cf. \cite[Theorem~19.6]{hio}) which is an
extension to the equimultiple case of D. Rees' multiplicity
criterion for primary ideals in terms of reductions of ideals
\cite{re}. Non primary equimultiple ideals may be produced, for
instance, via linkage as shown by A. Corso, C. Polini and W. V.
Vasconcelos in \cite{cpv}.

\smallbreak

 Multiplicity theory was extended by D. Buchsbaum and
D. S. Rim \cite{br} to submodules of finite colength in a free
module introducing what is known by Buchsbaum-Rim multiplicity,
while D. Rees introduced the theory of reductions and integral
closure of modules in \cite{dr}. In this context of modules, Rees'
multiplicity criterion was proven by D. Kirby and D. Rees in
\cite{kr} and by S. L. Kleiman and A. Thorup in \cite{kt} and D.
Katz gave the corresponding extension of B{\"o}ger's theorem to
equimultiple modules in \cite{dk}. Both Buchsbaum-Rim multiplicity
and integral closure of modules play an important role in the work
by T. Gaffney \cite{tg1, tg2} on the study of equisingularity
conditions of isolated complete intersection singularities (ICIS),
which has been an important source of motivation to pursue the
study of multiplicity theory and related topics in the context of
modules.

\smallbreak

Equimultiple modules have also been defined by A. Simis, B. Ulrich
and W. V. Vasconcelos in \cite{suv} as a particular class of ideal
modules: The class of ideal modules behaves somehow similarly to
the class of ideals and one is then able to define the analytic
deviation of an ideal module, the equimultiple modules being those
with analytic deviation zero. Their definition is slightly
different but agrees with ours in the \CM{} case. They also show
how to produce such modules via linkage.


\smallbreak

The main purpose of this paper is to make a systematic approach to
the properties of equimultiple modules by using the theory of
reductions of modules. As application we obtain several criteria
for an equimultiple module to be a complete intersection and prove
the openness of the equimultiple locus of an ideal module
extending to the case of modules the corresponding results in the
ideal case.

\smallbreak

Let $R$ be a \nor{} with total ring of fractions $Q$ and $E \sube
G \simeq R^{e}$ an $R$-module having rank $e>0$. In our context,
many of the structural properties of $E$ are reflected by the
quotient $G/E$ and by the $e$-th Fitting ideal $F_{e}(E)$, being
these two sets related by $V(F_{e}(E)) \sube \supp G/E$. Moreover,
in the case where $\grd G/E \geq 2$, $E$ is said to be an {\it
ideal module}, and the inclusion is then an equality (see
\reft{bI1}). The Fitting ideals play an important role in the
study of this class of modules, interviewing in the definition of
\ci{} and \eq{} modules, cf. section \ref{sec:dad}. For this
reason, we pay special attention to the relations between $G/E$
and $F_{e}(E)$, cf. section \ref{sec:supp}.

\smallbreak

Complete intersection modules (i. e. modules of the principal
class) are of course equimultiple. We then prove several
criteria for an \eq{} module to be a \ci{} extending to modules
the corresponding ones in the ideal case. For example:

\begin{theorem*}{\rm [cf. \reft{cieq3}]}
    Let $R$ be a \cm, $E$ a non-free \fgtfrmo having rank $e >0$.
    Suppose that $E$ is \gci. Then $E$ is \ci{} \sse $E$ is \eq.
\end{theorem*}


In section \ref{sec:open} we also prove the openness of the \ci{}
and the \eq{} locus, for ideal modules.

\begin{theorem*}{\rm [cf. \refts{openci}, \ref{openeq}]}
Let $R$ be a \nor{} and $E \subn G \simeq R^{e}$ an ideal module.
Then
\begin{enumerate}
\item[a)] $U_{ci}= \{ \frp \in \supp G/E \mid E_{\frp} \,
\textrm{is a \ci}\}$ is a (possibly empty)

open subset in $\supp G/E$. \item[b)] $U_{eq}= \{ \frp \in \supp
G/E \mid E_{\frp} \, \textrm{is \eq}\}$ is a non-empty open subset
in $\supp G/E$.
\end{enumerate}
\end{theorem*}

\smallbreak

As in the case of ideals, the notion of Rees algebra appears
naturally in this context. Let $R$ be a \nor{} and $E$ \fg{}
\rmo{} that affords an embedding into a free \rmo,
$E\overset{f}{\hra} G \overset{\varphi}{\simeq} R^{e}$. For such a
module, the Rees algebra $\RE$ of E is the $R$-subalgebra of the
polynomial ring $R[\sek{t}{e}]$ generated by all linear forms
$a_{1}t_{1} + \cdots + a_{e}t_{e}$, where $\seq{a}{e}$ is the
image of an element of $E$ in $R^e$ under the embedding $\varphi
\circ f$. Summarizing,
\begin{equation*}
  \RE : =\bop_{n \geq 0} \cS(f)_{n}(\cS(E)_{n})
  \sube R[\sek{t}{e}],
\end{equation*}
where $\apl{\cS(f)}{\cS(E)}{\cS(G)= R[\sek{t}{e}]}$ is the induced
map of symmetric algebras. One should note that, for given
different embeddings of $E$ into free $R$-modules, we can get non
isomorphic Rees algebras, see for instance A. Micali
\cite[Chapitre III, 2. Un example]{mi} or the more recent D.
Eisenbud, C. Huneke and B. Ulrich \cite[Example 1.1]{ehu}. See
also these papers for a discussion about the uniqueness of the
definition of the Rees algebra of a module.

\smallbreak

In the particular case that $E$ is a \fg{} torsionfree \rmo{} with
rank $e$, then $E$ affords an embedding into a free module of the
same rank, $E \overset{f}{\hra} G \overset{\varphi}{\simeq} R^{e}$
and one can see (because $E$ is torsion free) that
\begin{equation*} \label{frel}
   \RE \simeq \SE/\tau_{R}(\SE),
\end{equation*}
so the Rees algebra of $E$ is independent of the embedding $f$. We
then denote by $E^n$ the $n$-th graded piece of $\RE$, that is
$E^{n} := \RE_{n}$ and call it the $n$-th Rees power of $E$.

\smallbreak

A special case is the Rees algebra of a module $E=I_{1}\+ \cdots
\+ I_{e}$ where $\sek{I}{e}$ are $R$-ideals. Then, $\RE$ is the
multi-Rees algebra $\CR(\sek{I}{e})=R[I_{1}t_{1}, \ldots ,
I_{e}t_{e}]$. In section \ref{sec:ex} we give some examples of
equimultiple modules of this type. Finally, in section
\ref{sec:rp}, we characterize the non-free locus of the
corresponding Fitting ideal of each $n$-th Rees power $E^n$, and
give an easy proof of the Burch's inequality for \eq{} modules.

\smallbreak

In this paper we shall not use the notion of integral closure of
modules. For the general aspects of this theory we refer to the
corresponding chapters of the recent books by W. V. Vasconcelos
\cite{v2} and I. Swanson and C. Huneke \cite{sh}.

\section{Reduction of modules} \label{sec:rm}

\noindent In this section we review the notion of reduction of
modules and state the results we shall use throughout this paper.
\smallbreak

Suppose that $E$ is a \fgtfrmo having a rank over a \nor{} $R$.
Let $U$ be an $R$-submodule of $E$. $U$ is said to be a {\it
reduction} of $E$ if $$ E^{r+1} = U \cdot E^{r} $$ for some $r
\geq 0$ (this product taken inside $R(E)$). The least integer $r$
for which $E^{r+1}= U \cdot E^{r}$ is called {\it the reduction
number of} $E$ {\it with respect to} $U$, and is denoted by
$r_{U}(E)$. A reduction of $E$ is called {\it minimal} if it is
minimal with respect to inclusion. \smallbreak

It is clear that $E$ is a reduction of itself with $r_{E}(E)=0$.
Moreover, if $U$ is a reduction of $E$, then $U \oxr S$ is a
reduction of $E \oxr S$ where $S$ is any of the rings $R_{\frp}$
with $\frp$ a prime ideal, $Q=\Quot(R)$ or a polynomial ring.
Further if $U$ is a reduction
and $E^{r+1} = U \cdot E^{r}$ for some $r \geq 0$ then
$E^{n+1} = U \cdot E^{n}$ for all $n \geq r$. \smallbreak

Since $\RE$ is a standard graded algebra over $R$, one may also
apply to this situation the notion of reduction for graded rings
introduced by A. Ooishi in \cite{o}. In fact, this is equivalent
to the above notion of reduction of modules when the results in
\cite{o} are adequately read in our set up. Alternatively, it is
possible to translate to the case of modules the results and
proofs in \cite[section 10]{hio} for ideals in order to obtain the
basic properties of the theory of reduction of modules.
\smallbreak

Recall that given a \nol{} $\Rmk$ the {\it fiber cone} of $\RE$ is
the graded ring $\FE = \RE/\fm \RE = \bigoplus_{i\geq 0}
E^{i}/\fm E^{i}$. The Krull dimension of $\cF(E)$ is called
the {\it analytic spread} of $E$ and is denoted by $\lE$. For an
element $a\in E$ we denote by $\ol{a}= a+ \fm E \in E/\fm E\subset
\FE$. Then one can see that $U=Ra_1 + \cdots + Ra_n$ is a
reduction of $E$ \sse $\dim \FE/\langle \ol{a_1}, ... , \ol{a_n}
\rangle=0$. In particular, we get $\mu(U) \geq \lE$. \smallbreak

Next, we list the results on the theory of reduction of modules
that we shall use later in this paper. \smallbreak

\begin{proposition}\label{red}
    Let $\Rmk$ be a \nol, $E$ a \fgtfrmo having rank, $U$ a
    reduction of $E$.

    \begin{enumerate}
        \item  There exists $V\subseteq U\subseteq E$ which is a minimal reduction of $E$.

        \item  If $V\subseteq E$ is a minimal reduction of $E$ and
        $V=\langle a_1, \ldots ,a_n \rangle$ with $n=\mu(V)$, then
        $\sek{\ol{a}}{n} \in \FE$ are linearly independent, i.e. $\fm E\cap V=\fm  V$.

        \item If $V\subseteq U \subseteq E$ is a minimal reduction of $E$ and
        $V=\langle a_1, \ldots ,a_n \rangle$ with $n=\mu(V)$, then there exist
        $b_1, \ldots ,b_m \in E$ such that $\langle a_1, \ldots , a_n, b_1, \ldots , b_m \rangle =E$
        and $\mu (E)=n+m$. In particular, $\mu(U) \geq \mu (V)\geq \lE$.

        \item The following are equivalent:

        \begin{enumerate}

        \item[d1)] $V\subseteq E$ is a reduction and $\mu (V)=\lE$.

        \item[d2)] If $V=\langle a_1, \ldots ,a_n \rangle$ with
        $n=\mu(V)$, then $\sek{\ol{a}}{n} \in \FE$ is a
        homogeneous system of parameters.

        \end{enumerate}

        And if any of these two equivalent conditions holds, $V$
        is a minimal reduction of $E$.

        \item If the residue field $k$ is infinite and $V\subseteq U$ is a minimal reduction,
        then conditions d1) and d2) hold. In particular,
        $\FV\subset \FE$ is a noether normalization of $\FE$
        and $V^n \cap \fm E^n =\fm V^n$ for all $n\geq 0$.
    \end{enumerate}
\end{proposition}

As a consequence, minimal reductions always exist. The {\it
reduction number} of $E$, denoted by $r(E)$, is the minimum of
$r_{U}(E)$, where $U$ ranges over all minimal reductions of $E$.
\smallbreak

\begin{remark}\label{hsop}
If the residue field is finite, a minimal set of generators of a
minimal reduction of $E$ is not necessarily a homogeneous system
of parameters of $\FE$. Nevertheless, there always exist
homogeneous systems of parameters of $\FE$. This is equivalent to
the existence of a family of elements $a_1\in E^{r_1}\setminus \fm
E^{r_1}, \ldots , a_s\in E^{r_s}\setminus \fm E^{r_s}$, where
$s=\lE$, such that for some $r$, $E^r=a_1E^{r-r_1}+ \cdots +
a_sE^{r-r_s}$; and $\lE$ is the minimum positive number for a such
family of elements to exist.

\end{remark}

\begin{corollary}\label{as0}
    Let $\Rmk$ be a \nol, $E$ a \fgtfrmo
    having rank.
    Then
    $\ell(E_{\frp}) \leq \lE$ for all $\frp \in \sR$
    \end{corollary}

\begin{proof}
     Assume first that $k$ is infinite. Let $U$ be a minimal
     reduction of $E$ and let $\frp \in \sR$
     be any prime. Then $U_{\frp}$ is a reduction of $E_{\frp}$ and so
     $$ \ell(E_{\frp}) \leq \mu(U_{\frp}) \leq \mu(U)=\lE.$$
     Let $R\prd$ be a Nagata extension of $R$. Hence
     $\frq = \frp R\prd \in \spec(R\prd)$ and $R_{\frq}\prd$ is
     a Nagata extension of $R_{\frp}$. Therefore,
     applying the above inequality
     $$ \ell(E_{\frp}) = \ell(E_{\frp} \ox_{R_{\frp}} R_{\frq}\prd) =
     \ell(E_{\frq}\prd) \leq \ell(E\prd)=\lE.$$
     \end{proof}

A module $E$ is said to be of {\it linear type} if $\RE=\SE$. Clearly, every \fg{} free
module over a \nor{} is of linear type. \smallbreak

Next we observe that a module of linear type admits no proper reductions.


\begin{corollary}\label{as01}
    Let $R$ be a Noetherian ring, $E$ a \fgtfrmo
    having rank. If $E$ is of linear type then $E$ has no proper
    reductions. In particular, if $E$ is a free
    $R$-module then $E$ has no proper reductions.
\end{corollary}

\begin{proof}
    Assume first that $\Rmk$ is local. Then, we have
    $$ \lE= \dim(\RE \oxr k) = \dim(\SE \oxr k) = \dim \cS_{k}(E
    \oxr k).$$
    Since $E\oxr k$ is a free $k$-module, $\rank(E\oxr
    k)=\dim_{k}(E\oxr k)=\mu(E)$, and so $\lE= \mu(E)$.
    By \refp{red}, $E$ is a minimal reduction of itself.
    Hence $r(E)=0$.

    Let now $U\subseteq E$ be a reduction of $E$. Then, $U_{\fm} \subseteq E_{\fm}$
    is a reduction of $E_{\fm}$ for any maximal ideal $\fm\subset R$, and by the
    local case, $U_{\fm} = E_{\fm}$. Therefore, $U=E$.
\end{proof}

Any reduction $U$ of $E$ has rank and $\rank U = \rank E$. Namely,

\begin{proposition}\label{pred}
    Let $R$ be a \nor, $E$ a \fgtfrmo having rank and $U \sube E$
    be a reduction of $E$.
    \begin{enumerate}
        \item  $U$ has rank and $\rank U = \rank E$;

%
%
%
        \item  $\grd E/U >0$.


    \end{enumerate}
\end{proposition}

\begin{proof}
a)  If  $\rank E = e$ and $Q= \Quot(R)$, then $E \oxr Q \simeq
Q^{e}$ and by \refc{as01} $E \oxr Q = U \oxr Q$, proving that
$\rank E = \rank U$.


%

b)  We have $(E/U)_{\frp} \simeq E_{\frp}/U_{\frp} = 0$
    for all $\frp \in \Ass R$. But
    $$   (E/U)_{\frp}=0  \LRa \frp \not \in \supp E/U =
    V(\ann_{R}(E/U)) \LRa \ann_{R}(E/U) \nsubseteq \frp. $$
    It follows that $\ann_{R}(E/U) \nsubseteq \bigcup_{\frp \in \Ass
    R}\frp = \z(R)$, and so $\grd E/U >0$.
%
\end{proof}

We close this section mentioning the upper and lower bounds for
the analytic spread obtained in \cite{suv} and deducing two easy
consequences.

\begin{proposition} [{\cite[Proposition~2.3]{suv}}]\label{bas}
    Let $\Rmk$ be a \nol{} of dimension $d >0$ and $E$ a \fgr{}
    having rank $e$. Then $e \leq \lE \leq d + e -1$.
\end{proposition}

\begin{corollary} \label{bas1}
    Let $\Rmk$ be a \nol{} of dimension $d >0$ with infinite residue
    field and let $E$ a \fgtfrmo having rank $e$. Then $\lE =e$ \sse
    any minimal reduction of $E$ is a free \rmo.
\end{corollary}

\begin{proof}
    Let $U$ be a minimal reduction of $E$. According to
    \refp{red} and to \refp{bas}
    $$ \mu(U) = \ell(E) \geq e = \rank U.$$
    Hence $\lE =e$ \sse $\rank U = \mu(U)$ \sse $U$ is a free \rmo.
\end{proof}

\begin{corollary}\label{bas2}
    Let $\Rmk$ be a \nol{} of dimension $1$ with infinite residue
    field and let $E$ a \fgtfrmo having rank $e$. Then any minimal
    reduction of $E$ is a free \rmo.
\end{corollary}

\smallbreak
\section{The support of $G/E$ and ideal modules}\label{sec:supp}

\noindent Given a \fgtfrmo $E$ having positive rank $e$, $E$
affords an embedding into a free module of the same rank, $E \sube
G \simeq R^e$. The aim of this section is to realize the support
of $G/E$ as the variety of a certain ideal. To do this we first
establish the inclusions
$$ \supp E/U \sube V(F_{e}(E)) \sube \supp G/U,$$
where $F_{e}(E)$ is the $e$-Fitting ideal of $E$ and $U$ is a
reduction of $E$. As we will see, equality in the second inclusion
holds when $\grd G/E \geq 2$,
which determine a class of modules already introduced in \cite{suv} that
afford a natural embedding into a free module of the same rank,
the class of ideal modules. This type of modules behave similarly
to an ideal, as we show. \medbreak

Recall that $F_{i}(E):=I_{n-i}(\varphi)$ is the ideal generated by
the $(n-i)\x (n-i)$ minors of $\varphi$ where $R^m
\overset{\varphi}{\ra} R^n \ra E \ra 0$ is a finite presentation
of $E$, $0 \leq i \leq n$. If $E$ has positive rank $e$ then
$$ F_{0}(E)= \cdots = F_{e-1}(E)=(0) \subn F_e(E) \cdots \sube F_{i}(E) \sube \cdots \sube R.$$
Moreover, $\supp E = V(F_0(E))$ and, for every $\frp \in \sR$,
$\mu(E_{\frp})=n$ \sse $F_{n-1}(E) \sube \frp$ and $F_{n}(E) \nsub
\frp$. \medbreak

Using these properties and the fact that, when $\Rm$ is local, a
finitely generated module $E$ with rank $e>0$ is free if and only
if $\mu (E)=e$, one immediately gets that if $R$ is a Noetherian
ring and $E$ a \fg{} \rmo{} having rank $e >0$, then the free
locus of $E$ is given by $\sR \setminus V(F_{e}(E))$, and
coincides with the locus of prime ideals $\frp \in \sR$ such that
$\mu(E_{\frp}) \leq e$. In particular, if $\Rm$ is also a local
ring, then $ E \text{ is free \sse } F_{e}(E) = R$. \smallbreak

Applying the above facts we observe the following about $\grd
F_e(E)$. For its proof just use that $\grd F_{e}(E) = \inf \{
\depth R_{\frp} \mid \frp \in V(F_{e}(E)) \}>0$. \smallbreak

\begin{lemma}\label{g2}
    Let $E$ be a \fg{} module over a \nor{} $R$ having positive rank
    $e$. Then:
    \begin{enumerate}
        \item  $\grd F_{e}(E) > 0$.

        \item  $\grd F_{e}(E) \geq 2$ \sse $\mu(E_{\frp}) \leq e$ for every
        $\frp \in \sR$ whenever $\depth R_{\frp}= 1$.

        \item  If $\grd F_{e}(E) \geq 2$ then $E_{\frp}$ is free whenever $\dpt R_{\frp} \leq 1$.
    \end{enumerate}
\end{lemma}

Next we prove that $\grd E/U \geq \grd F_{e}(E)$ for any reduction
$U$ of $E$, $e= \rank E >0$.

\begin{proposition}\label{nfl0}
     Let $R$ be a \nor, $E$ a \fg{} torsionfree \rmo{} having rank $e >0$ and $U$
     a reduction of $E$. Then $\supp E/U \sube V(F_{e}(E))$.
\end{proposition}

\begin{proof}
    Let $\frp \not \in V(F_{e}(E))$.
    Hence $E_{\frp}$ is free and, since
    $U_{\frp}$ is a reduction of $E_{\frp}$, we must have $U_{\frp} =
    E_{\frp}$ (because a free module
    has no proper reductions). Therefore $\frp \not \in \supp E/U$.
    It follows that $\supp E/U \sube V(F_{e}(E))$.
\end{proof}

Let $E$ be a \fgtfrmo having rank $e >0$.
Hence $E$ is a submodule of a free
$R$-module $G \simeq R^e$, and so for any reduction
$U$ of $E$, $U \sube E \sube G$. In this case, we show that
$V(F_{e}(E)) \sube \supp G/U= \supp G/E$ with equality if $\grd G/U \geq 2$.

\begin{proposition}\label{nUfl}
    Let $R$ be a \nor, $E$ a \fgtfrmo having rank $e >0$
    and $U$ a reduction of $E$. Suppose that $E \sube G \simeq R^e$.
    Then
    \begin{enumerate}
        \item  $\supp E/U \sube V(F_{e}(E)) \sube \supp G/U = \supp G/E$;

        \item  $\grd E/U \geq \grd F_{e}(E) \geq \grd G/U = \grd G/E$.
    \end{enumerate}
\end{proposition}

\begin{proof}
    Let $U$ be a reduction of $E$.
    We first note that if $U=G$ then $U=E=G$, and so $G/U=0$ and
    $F_{e}(E)=R$. In this case, the formula reads
    $V(R)=\supp(0) = \emptyset$. Hence, we may assume that $U \neq G$.
    Let $\frp \in \sR$ be arbitrary. Since
    \begin{align*}
        \frp \not \in \supp G/U & \iff G_{\frp}/U_{\frp}=
        (G/U)_{\frp}=0 \iff U_{\frp} = E_{\frp} = G_{\frp} \\
        & \implies E_{\frp} \text{ is free } \iff F_{e}(E)
        \nsubseteq \frp,
    \end{align*}
    then $V(F_{e}(E)) \sube \supp G/U$. We also get $V(F_{e}(E)) \sube \supp
    G/E$, hence
    $$ \supp E/U \sube V(F_{e}(E)) \sube \supp G/E. $$
    Now from the exact sequence $0 \ra E/U \ra G/U \ra G/E \ra 0$
    we have
    $$\supp G/U = \supp E/U \cup \supp G/E = \supp G/E,$$
    and a) is proved.
    b) follows as a direct consequence of a).
\end{proof}

In the situation of \refp{nUfl}, we observe that if $E_{\frp}$ is free then
$E_{\frp} \simeq R_{\frp}^e \simeq G_{\frp}$; however, in general,
we may have $E_{\frp} \subn G_{\frp}$. Next we give a sufficient
condition to guarantee the equality.

\begin{lemma}\label{grd>2}
    Let $R$ be a \nor, $F \sube G$ \fg{} free \rmo s having the same
    rank. If $\grd G/F \geq 2$ then $F=G$.
\end{lemma}

\begin{proof}
    Suppose that $F \subsetneq G$. Hence $\Ass G/F \neq \emptyset$.
    Let $\frp \in \Ass G/F$. Since $\grd G/F$ $\geq 2$ then $\dpt
    R_{\frp} \geq 2$. Supposing $\rank F = e = \rank G$ then
    $F_{\frp} \simeq R_{\frp}^e \simeq G_{\frp}$, and so
    $$ \dpt F_{\frp} = \dpt G_{\frp} = \dpt R_{\frp} \geq 2.$$
    It follows that
    $$ \dpt G_{\frp}/F_{\frp} \geq \min \{ \dpt G_{\frp}, \dpt
    F_{\frp} -1 \} \geq 1,$$
    contradicting $\dpt
    G_{\frp}/F_{\frp} =0$. Hence $\Ass G/F=\emptyset$,
    and so $F=G$.
\end{proof}

\begin{proposition}\label{nfl}
    Let $R$ be a \nor, $E$ a \fgtfrmo having rank $e >0$
    and $U$ a reduction of $E$. Suppose that $E \sube G \simeq R^e$.
    If $\grd G/U \geq 2$ then
    \begin{enumerate}
        \item  $V(F_{e}(E)) = \supp G/U$;

        \item  $\grd G/U = \grd F_{e}(E)$.
    \end{enumerate}
\end{proposition}

\begin{proof}
    a) By \refp{nUfl}, $V(F_{e}(E)) \sube \supp G/U$. For the other
    inclusion let $\frp \in \sR \setminus V(F_{e}(E))$. Hence
    $E_{\frp}$ is free, and so
    $U_{\frp} = E_{\frp} \simeq
    R_{\frp}^e \simeq G_{\frp}$. By assumption
    $$ \grd G_{\frp}/U_{\frp} = \grd (G/U)_{\frp} \geq \grd G/U
    \geq 2. $$
    Hence, by the previous lemma, $U_{\frp}=G_{\frp}$ and so $\frp
    \not \in \supp G/U$. The equality holds.
    b) follows by a).
\end{proof}

\begin{theorem}\label{bI1}
    Let $R$ be a \nor, $E$ a \fgtfrmo having rank $e >0$.
    Suppose that $E \sube G \simeq R^e$.
    The following are equivalent:
    \begin{enumerate}
        \item  $\grd G/E \geq 2$;

        \item  $\grd G/U \geq 2$ for any reduction $U$ of $E$;

        \item  $\grd G/U \geq 2$ for some reduction $U$ of $E$.
    \end{enumerate}
    If any of the above conditions holds, $V(F_{e}(E))=V(F_{e}(U))=\supp G/E =$ \linebreak  $\supp G/U$ for any reduction $U$ of $E$; in particular, $\grd G/E = \grd F_{e}(E) = \grd F_{e}(U) = \grd G/U$.
\end{theorem}

\begin{proof}
    This follows by \refp{nUfl} and  by \refp{nfl}
    (applied twice).
\end{proof}

In general, the class of modules of the form $E \sube G \simeq
R^e$ with $\grd G/E \geq 2$ is sufficiently special to have a
name: {\it ideal module}. This definition of ideal module is one
of the various characterizations of ideal module in \cite
[Proposition~5.1-c)]{suv} of Simis-Ulrich-Vasconcelos, where ideal modules
are defined as the finitely generated and torsion free $R$-modules $E$, such that the
double dual $E^{**}$ is free. These type
of modules behave similarly to an ideal, because they afford a
natural embedding into a free module of the same rank, its bidual.
See \cite{suv} for details.

\smallbreak

It is worthwhile to point out that although the definition of ideal module is intrinsic and does not depend on the possible embedding of $E$ into a free module $G$, the property $\grd G/E \geq 2$ depends on the chosen embedding as the following simple example shows: Let $R=k[[x, y, z]]$ where $k$ is a field and $E=(zx,zy)$. Then $E\simeq I=(x,y)$ (as $R$-modules) and so $E^{**}\simeq I^{**}$ which is free because $I$ is an ideal of grade $2$. Thus $E$ is an ideal module. On the other hand, $\grd R/E=1$ because $E$ is an ideal of grade $1$. In this case, the "right" embedding for $E$ is given by $E\simeq I=(x,y)\subset G:=R$.

\smallbreak

It is also clear that any reduction $U$ of an ideal module $E$
having rank $e$ is an ideal module having rank $e$. Moreover, by
\refp{nUfl}, $\grd E/U\geq 2$.

\smallbreak

Ideal modules satisfy the following properties which are easy to prove.

\begin{remark}\label{idmo1}
    Let $R$ be a \nor, $\frp \in \sR$, $E$ an ideal module, and $G$ a \fg{} free module containing $E$ with
    $\grd G/E \geq 2$. Then
    \begin{enumerate}
        \item  $E$ has rank and $\rank E= \rank G$;

        \item  $\grd F_{e}(E)\geq 2$ , where $e = \rank E$;

        \item  $E_{\frp}$ is an ideal module;

        \item  $E_{\frp}$ is free whenever $\dpt R_{\frp} \leq 1$.

    \end{enumerate}
\end{remark}

For any reduction $U$ of an ideal module $E$ having rank $e$
contained in a free module $G\simeq R^e$ we have natural
isomorphisms $U^{**}\simeq E^{**}\simeq G^{**}\simeq G$. To see
that one just need to apply the following lemma:

\begin{lemma}\label{idmG}
   Let $R$ be a \nor{} and $E_2 \sube E_1$ finitely generated $R$-modules such that
   $\grd E_1/E_2 \geq 2$. Then, $E_2^{**}\simeq E_1^{**}$.
\end{lemma}

\begin{proof}
    Dualizing the exact sequence
       $0 \ra E_2 \hra E_1 \tra E_1/E_2 \ra 0$
    we obtain the exact sequence
       $$ 0 \ra (E_1/E_2)^{*} \ra E_1^{*} \ra E_2^{*} \ra \Ext_{R}^1(E_1/E_2,R).$$
    Since $\grd E_1/E_2 \geq 2$,
       $ (E_1/E_2)^{*} = \Ext_{R}^1(E_1/E_2,R)=0,$
    and so $E_1^{*} \simeq E_2^{*}$. Therefore $E_2^{**} \simeq
    E_1^{**}$.
    \end{proof}

\begin{proposition}\label{inc}
    Let $R$ be a \nor{} and $E$ an ideal module of rank $e$. Then all free $R$-modules $R^e \simeq G_{i} \supseteq E$ with $\grd G_{i}/E \geq 2$ are incomparable for inclusion and $G_{i} \simeq E^{**}$.
\end{proposition}

\begin{proof}
Suppose that $E \subset G_{i} \sube G_{j}$ with $G_{i}, G_{j} \simeq R^e$, $\grd G_{i}/E \geq 2$, $\grd G_{j}/E \geq 2$. Hence $\grd G_{j}/G_{i} \geq 2$, and so $G_{i}=G_{j}$ by \refl{grd>2}. The last assertion follows by \refl{idmG}.
\end{proof}

Next we observe that over a \nol{} $\Rm$ with $\dpt R \geq 2$,
if $\dim G/E =0$ then $\grd G/E \geq 2$ and $E$ is an ideal module.
In particular, any $\fm$-primary $R$-ideal $I$ will be an ideal
module.

\begin{proposition}\label{dge0}
    Let $R$ be a \nol{} with
    $\dpt R \geq 2$ and let $E \subsetneq G \simeq R^e$ be an \rmo{}.
    If $\dim G/E =0$ then $\grd G/E \geq 2$. In par\-ti\-cu\-lar
    $E$ is an ideal module.
\end{proposition}

\begin{proof}
    By assumption $\dim G/E=0$ and $\dpt R \geq 2$. Hence
    $$\Hom_{R}(G/E,R)=0=\Ext^{1}_{R}(G/E,R),$$
    by \cite[Theorem~17.1]{mat}.
    Thus $\grd G/E \geq 2$, and so $E$ is an
    ideal module.
\end{proof}

In the following, we determine $\dim G/E$ and $\dpt G/E$. In
particular, we observe that any ideal module has maximum Krull
dimension.

\begin{proposition}\label{dime}
    Let $R$ be a \nol, $\dim R = d \geq 2$, $E$ an ideal module over
    $R$ and $U$ a reduction of $E$. Suppose that
    $E \subn G \simeq R^e$, $e >0$. Then
    \begin{enumerate}
        \item  $\dim G/E \leq d - \Ht F_{e}(E)\leq d-2$;

        \item $ \dim E = d$;

       \item  if in addition $R$ is \CM, then
    $$\dpt E -1 = \dpt G/E \leq \dim G/E = d- \Ht F_{e}(E).$$
    \end{enumerate}
\end{proposition}

\begin{proof}
    a) We have  $\supp G/E = V(F_{e}(E))= \supp R/F_e(E)$ and so
         $$ \dim G/E = \dim R/F_e(E) \leq d - \Ht F_e(E) \leq d-2.$$

    b) Since $G$ is free and $\dim G/E <d$, then
    $$ d = \dim R = \dim G = \max \{ \dim E, \dim G/E \} = \dim E,$$
    as required.

    c) For the first equality we apply the depth Lemma to the exact sequence
    $0 \ra E \ra G \ra G/E \ra 0$. Now, c) follows by a).
\end{proof}


As stated in Proposition \ref{bas}, the analytic spread of a \fg{}
module $E$ having rank $e$, over a $d$-dimensional \nol, satisfies
the inequalities
$$ e \leq \lE \leq d+e-1. $$
Now we deduce another lower bound for the analytic spread, for any
\tf module, and as a consequence we recover the one stated in
\cite[Proposition~5.2]{suv} in the case where $E$ is, in addition,
an ideal module.

\begin{proposition}\label{lub}
    Let $R$ be a \nol{} and $E \sube G \simeq R^e$ a
    \fgtfrmo having rank $e >0$, but not free. Then
    $$ \ell(E) \geq \Ht F_{0}(G/U) +e-1,$$
    for any minimal reduction $U$ of $E$.
    In particular, if $\grd G/E$ $\geq 2$
    then $\ell(E) \geq \Ht F_{e}(E) +e-1 \geq e+1.$
\end{proposition}

\begin{proof}
    We may assume that the residue field of $R$ is infinite, since
    any Nagata extension $R\prd$ of $R$ has infinite residue field
    and, for any \fg{} \rmo{} $M$,
    $\ell(M \oxr R\prd) = \ell(M)$,
    $\rank(M \oxr R\prd) = \rank M$ and $\Ht F_{i}(M \oxr R\prd) =
    \Ht F_{i}(M)R\prd = \Ht F_{i}(M)$.

    Let $U$ be a minimal reduction of $E$
    and suppose that $\mu(U)=n$ (hence $\ell(E) = n$).
    Then there exists an $R$-epimorphism $\apl{\psi}{R^n}{U}$.
    Further, since $E$ is not free,
    $U$ is a (proper) submodule of $G$.
    Therefore, we have an exact sequence
    \begin{equation}\label{esgu}
        R^n \overset{\psi}{\ra} G \ra G/U \ra 0.
    \end{equation}
    By the Eagon-Northcoot Theorem (see
    \cite[Theorem~1.1.16]{v1}),
    $$ \Ht F_{0}(G/U) = \Ht I_{e}(\psi) \leq n-e +1 = \ell(E)
    -e+1,$$
    proving the inequality. Moreover, if $\grd G/E \geq 2$ then
    $\Ht F_{0}(G/U) = \Ht F_{e}(E) \geq 2$, (by \reft{bI1}),
    and the other inequalities follow.
\end{proof}
\smallbreak


If $\Rm$ is a \nol{} and $I$ an $\fm$-primary ideal then the
analytic spread is the biggest possible: $\ell(I) = \dim R$. Let
$E$ be a \fg{} \rmo{} having rank $e>0$, but not free. Since the
free locus of $E$ is given by $\spec (R)\setminus V(F_e(E)$ we
have that $E$ is free locally on the punctured spectrum, that is
$E_{\frp}$ is free for every prime $\frp \neq \fm$, if and only if
$F_e(E)$ is an $\fm$-primary ideal. As a consequence, we get the
formula given in \cite[Proposition~5.2]{suv} for ideal modules
which are \flps.

\begin{corollary}\label{vbm1}
    Let $\Rm$ be a \nol{} of dimension $d$ and let
    $E$ be an  ideal module having rank $e>0$ which is \flps.
    Then $\ell(E) = d+e-1= \Ht F_{e}(E)+e-1.$
\end{corollary}

We note here that to be free locally on the punctured spectrum is not a sufficient condition for a module to be an ideal module, as the following simple example shows: Let $R=k[[x, y, z]]$ where $k$ is a field and $E=Re_1\oplus Re_2\oplus Re_3/(xe_1+ye_2+ze_3)$, where $e_1, e_2, e_3$ is the canonical basis of $R^3$. Then, $\rank E=2$, $F_2(E)=(x,y,z)$ and $E$ has projective dimension $1$. Thus $E$ is free locally on the punctured spectrum. On the other hand, by \cite[Proposition 1.4.1]{bh} $E$ is reflexive and so $E^{**}$ is not free.

\smallbreak

In the case where $R$ is a \nol{} with $\dpt R \geq 2$ we proved
that if $\dim G/E=0$ then $E$ is an ideal module (cf.
\refp{dge0}). Thus we have the following equivalence.

\begin{proposition}\label{cvb}
    Let $\Rm$ be a \nol{} with $\dpt R \geq 2$ and let
    $E \subsetneq G \simeq R^e$ be an \rmo{} having rank $e>0$.
    Then $E$ is \flps{} with $\grd G/E \geq 2$
    \sse $\dim G/E=0$.
\end{proposition}

\begin{proof}
    Suppose that $E \subsetneq G \simeq R^e$ is \flps{}
    and that $\grd G/E \geq 2$. Hence $E_{\frp} \simeq R_{\frp}^e
    \simeq G_{\frp}$ for each prime $\frp \neq \fm$. Since $\grd
    G_{\frp}/E_{\frp} \geq \grd G/E \geq 2$ then $E_{\frp} =
    G_{\frp}$ (by \refl{grd>2}).
    Therefore $\supp G/E =\Ass G/E =\{ \fm \}$, and so $\dim G/E = 0$.
    The converse follows by \refp{dge0}.
    \end{proof}

\begin{corollary}\label{vbm11}
    Let $\Rm$ be a \nol{} with $\dpt R \geq 2$ and let
    $E$ be an  ideal module having rank $e>0$ which is \flps.
    Then any reduction $U$ of $E$ is \flps.
\end{corollary}

\begin{proof}
    Suppose that $E \subsetneq G \simeq R^e$ with $\grd G/E \geq 2$.
    Let $U$ be a reduction of $E$. Hence
    $$\supp G/U= V(F_{e}(E)) = \supp G/E =\{ \fm \}$$
    (by \reft{nfl}, \refp{dime}).
    Therefore $\dim G/U =0$ and $U$ is \flps.
\end{proof}

In the case of dimension $2$ every ideal module is \flps.

\begin{corollary}\label{dime2}
    Let $R$ be a \cm, $\dim R = 2$, $E \subn G \simeq R^e$, $e >0$ an ideal module over $R$.
    Then $E$ is \flps.
\end{corollary}

\begin{proof}
    By \refp{dime}, $\dim G/E=0$, and the assertion follows by \refp{cvb}.
\end{proof}

\smallbreak
\section{Deviation and analytic deviation}\label{sec:dad}

\noindent
In this section we define the deviation and the analytic deviation
for a module. These invariants give rise to the notions of \ci, \eq,
and \gci{} for modules, as in the case of ideals.
\medbreak

Let $\Rm$ be a \nol{} and $I$ an $R$-ideal. Recall that the {\it
deviation} of $I$ is defined to be the difference $\de(I) = \mu(I)
- \Ht I$, whereas the {\it analytic deviation} of $I$ is the
difference $\ad(I) = \ell(I) - \Ht I$. We always have $\de(I) \geq
0$ (by Krull's Principal Ideal Theorem) and $\ad(I) \geq 0$. As a
matter of fact, we have
$$ \mu(I) \geq \ell(I) \geq \Ht I $$
(cf. \cite[Proposition~10.20]{hio}). In the case where $\de(I)=0$,
$I$ is called a {\it complete intersection} and if $\ad(I)=0$, $I$
is said to be {\it equimultiple}. Furthermore, if $\mu(I_{\frp}) =
\Ht I$ for all minimal prime ideals $\frp \in \Min R/I$, $I$ is
called {\it generically a complete intersection}. \smallbreak

For non-free modules we have the following definitions.

\begin{definition}\label{def:dad}
    Let $R$ be a \nol{} and $E$ a \fgr{} having rank $e>0$ but not free. We
    define the {\it deviation} of $E$ by $\de(E) = \mu(E) -e+1 - \Ht
    F_{e}(E)$ and the {\it analytic deviation} of $E$ by
    $\ad(E) = \ell(E) -e +1 - \Ht F_{e}(E)$.
\end{definition}

We notice that our definitions slightly differ from those in
\cite{suv}, since we use $\Ht F_{e}(E)$ instead of $\grd
F_{e}(E)$. Clearly, they coincide in the \CM{} case. \smallbreak

Applying  \refp{red} and \refp{lub} we get the following.

\begin{remark}\label{des}
    Let $R$ be a \nol{} and $E$ an ideal module having rank $e$,
    but not free. Then $ \de(E) \geq \ad(E) \geq 0.$
\end{remark}

In accordance with the previous remark we have the following
definitions for non-free ideal modules.

\begin{definition}\label{eciaci}
    Let $R$ be a \nol{} and $E$ a non-free ideal module over $R$
    of rank $e$.  We say that $E$ is:
    \begin{enumerate}
    \item[1.]  a {\it \ci{} module} if $\de(E)=0$,

    \item[2.]  an {\it \eq{} module} if $\ad(E)=0$,

    \item[3.]  {\it \gci{} module} if $\mu(E_{\frp}) = \Ht F_{e}(E) +e-1$
    for all minimal prime ideals $\frp \in \Min R/F_{e}(E)$.
    \end{enumerate}
\end{definition}

As expected, as in the case of $\fm$-primary ideals we have the
following example of \eq{} modules.

\begin{example}\label{eeq}
    Let $R$ be a \nol{} with $\dim R = d >0$ and $E$ be a non-free ideal module which is \flps.
    Then, by \refc{vbm1} $E$ is \eq.
    \end{example}

Complete intersection modules were already considered by D.
Buchsbaum and D. Rim \cite{br} and by D. Katz and C. Naude
\cite{kn}, in particular situations. In fact, Katz-Naude defined a
{\it module of principal class} $E \sube R^{e}$ if it is generated
by $n \geq e$ column vectors and $\Ht F_{e}(E)=n-e+1$. If, in
addition,  $R$ is a local ring and $E$ is embedded into a free
module $G$ such that the quotient $G/E$ has finite length, then
the $e \times n$ matrix whose columns correspond to the generators
of $E$ is a parameter matrix in the sense of \cite{br}, and $E$ is
called a {\it parameter module}.\smallbreak

Clearly, if $R$ is a local ring, an ideal module is of principal
class \sse is a \ci. Moreover, in virtue of \refp{cvb}, if $\depth
R\geq 2$ any non-free parameter module having positive rank is a
\ci{} and also \flps. \medbreak

As in the case of ideals we have the following relations, that we
list here for completeness.

\begin{proposition} \label{pdf}
     Let $\Rmk$ be a \nol{} with $\dim R = d >0$ and $E$ a non-free
     ideal module having rank $e$.
     \begin{enumerate}
        \item  If $E$ is a \ci{} then:
        \begin{enumerate}
            \item  $E$ is \eq;

            \item  $E$ is \gci;

            \item  $\Ht F_{e}(E_{\frp}) = \Ht F_{e}(E)$,
           $\mu(E_{\frp})=\mu(E)$ for every $\frp \in \sR$;

            \item  $E_{\frp}$ is a \ci{} for every $\frp \in \sR$.
        \end{enumerate}

        \item  $E$ is \gci{} \sse
        \begin{enumerate}
           \item  $E_{\frp}$ is a \ci{} for every $\frp \in \Min R/F_{e}(E)$, and

           \item $\Ht F_{e}(E_{\frp}) = \Ht F_{e}(E)$ for every $\frp
           \in \Min R/F_{e}(E)$;

        \end{enumerate}

        \item  If there exists a reduction $U$ of $E$ which is a \ci{}
        then $E$ is \eq.

        \item  If $k$ is infinite, then $E$ is \eq{} \sse
        every minimal reduction $U$ of $E$ is a \ci.

        \item  If $E$ is a \ci{} then $E$ is \eq, $E$ admits no
        proper reductions and $r(E)=0$. If $k$ is infinite also the converse holds.

        \item  Suppose that $E$ is \flps.
        Then $E$ is a \ci{} \sse $\mu(E)=d+e-1$.
     \end{enumerate}
\end{proposition}

\begin{proof}
a)  The first assertion is immediate by \refr{des}.  For the
others,
    let $\frp \in \sR$ be arbitrary.
    Then we have,
    $$ \mu(E_{\frp})-e+1 \geq \Ht F_{e}(E_{\frp}) \geq \Ht F_{e}(E)
    = \mu(E) -e+1 \geq \mu(E_{\frp}) -e+1, $$
    and so (ii)-(iv) hold.

b) Suppose that $E$ is \gci{} module. Let $\frp \in
    \Min$ $R/F_{e}(E)$.
    Hence, $E_{\frp}$ is an ideal module over $R_{\frp}$
    having rank $e$, and we have
    $$ \Ht F_{e}(E) \leq \Ht F_{e}(E_{\frp}) \leq \ell(E_{\frp})
    -e+1 \leq \mu(E_{\frp})-e+1 = \Ht F_{e}(E).$$
    Therefore,
    $\Ht F_{e}(E) = \Ht F_{e}(E_{\frp}) =\mu(E_{\frp})-e+1$
    and $E_{\frp}$ is a \ci. The converse is clear.
%

c)  Let $U$ be a reduction of $E$ which is a \ci{} module. Since $E$ is an ideal
    module we have, by \reft{bI1} and by \refp{red}
    $$ \Ht F_{e}(E)= \Ht F_{e}(U) = \mu(U) -e+1 \geq \lE -e+1 \geq
    \Ht F_{e}(E),$$
    proving that $E$ is \eq.

d)  follows by \reft{bI1} and by \refp{red}.

e)  Suppose that $E$ is a \ci. By a) $E$ is \eq. Moreover,
    $$ \Ht F_{e}(E) \leq \ell(E) -e+1 \leq \mu(E) -e+1 = \Ht
    F_{e}(E),$$
    and so $\mu(E)=\ell(E)$. Hence $E$ is a minimal reduction of itself (cf.
    \refp{red}).
    For the converse we have, by assumptions,
    $ \mu(E)=\lE=\Ht F_{e}(E)+e-1,$
    proving that $E$ is a \ci.

f) follows by \refc{vbm1}.
\end{proof}

We may construct \ci{} [resp. \eq{} or \gci] modules of any rank
$e \geq 2$ using ideals of the same type. First we note that if
$E\simeq F\oplus E'$ is a finitely generated torsionfree
$R$-module having rank, where $F$ is a free $R$-module of rank
$e$, then $\RE \simeq \CR(E')[t_1, \ldots ,t_e]$.

\begin{corollary}\label{pdf1}
    Let $\Rm$ be a \nol{} with $\dim R = d >0$. Suppose that $E = F \+ I$ where
    $F \simeq R^{e-1}$ and $I$ an $R$-ideal with $\grd I \geq 2$.
    Then $E$ is a \ci{} [resp. \eq{} or \gci] module \sse
    $I$ is a \ci{} [resp. \eq{} or \gci] ideal.
\end{corollary}

\begin{proof}
We have $V(F_{e}(E))=V(F_{1}(I))=V(I)$, hence $\Ht F_{e}(E)=\Ht
I$. Moreover, $\mu(E)=\mu(I)+e-1$ and $\lE= \dim \FE = \dim
\FI[\sek{t}{e-1}]= \lI +e-1$. It follows that $E$ is a \ci{} [\eq]
module \sse $I$ is a \ci{} [\eq] ideal. For \gci{} modules apply
\refp{pdf}.
\end{proof}

\smallbreak
\section{Equimultiple versus \ci}\label{sec:eqci}

Complete intersection modules have good properties. In fact,
Simis-Ulrich-Vas\-con\-celos showed that, in this case, $\RE$ is
\CM{} (\cite[Corollary~5.6]{suv}) and Katz-Naude had proved that
$G/E$ is a perfect module (\cite[Proposition~3.3]{kn}). Hence if
$R$ is \CM{} and $E$ a \ci{} \rmo{} such that that $E \subn G
\simeq R^{e}$, then $G/E$ is \CM{} and $\grd G/E= \pjd G/E = \Ht
F_e(E)$.

\smallbreak We establish now few additional properties and prove
several criteria for an \eq{} module being a \ci{}.

%
%

\smallbreak We observed that $\lE \geq \Ht F_{0}(G/U)+e-1$ for any
non-free \rmo{} $E \subsetneq G \simeq R^e$ of rank $e>0$ and any
minimal reduction $U$ (cf. \refp{lub}). In the case where $E$ is
equimultiple and $\grd G/E \geq 2$ it is clear that the equality
holds. Moreover, in this case, if $R$ a \cm, we show that
$F_{0}(G/U)$ is a perfect ideal and all the associated primes of
$F_{0}(G/U)$ have the same height which is equal to $\lE -e + 1=
\Ht F_{e}(E)$.

\begin{proposition}\label{eci1}
    Let $R$ be a \cm, $\dim R = d > 0$ and $E$ a \ci{} module having rank
    $e>0$. Suppose that $E \subsetneq G \simeq R^e$. Then
    \begin{enumerate}
        \item  $F_{0}(G/E)$ is a perfect ideal;

        \item  $\dpt G/E = \dpt R/F_{0}(G/E)$;

        \item  $\Ass G/E = \{ \frp \in V(F_{e}(E)) \mid \Ht \frp = \lE -e +1 \}
        = \Ass R/F_{0}(G/E) = \Min R/F_{0}(G/E) = \Min R/F_{e}(E)$.
     \end{enumerate}
\end{proposition}

\begin{proof}
    a), b) We may assume that the residue field of $R$ is infinite.
    Suppose that $n = \mu(E)$ ($=\lE$).
    Since $E$ is \ci{} (hence ideal module) and $R$ is \CM,
    $$ n-e+1 = \Ht F_{e}(E) =\Ht F_{0}(G/E) = \grd F_{0}(G/E) $$
    -the second equality by \refp{nfl}.
    Hence, by \cite[Theorem~2.7]{bv}, $F_{0}(G/E)= I_{e}(\psi)$,
    with $\psi$ as in (\ref{esgu}), is a perfect ideal.
    Therefore,
    by the Auslander-Buchsbaum formula and since $R$ is \CM,
    \begin{align*}
        d- \dpt G/E &= \lE -e+1 = \pjd R/F_{0}(G/E) \\
        & = d - \dpt R/F_{0}(G/E)
    \end{align*}
    and b) follows.

    c) By \refp{nfl}, $\Min R/F_{0}(G/E)= \Min R/F_{e}(E)$.
    On the other hand, $R/F_{0}(G/E)$ is a \cm{} (by \cite[Theorem~2.1.5]{bh}).
    Thus \linebreak $\Min R/F_{0}(G/E) = \Ass  R/F_{0}(G/E)$.
    Since $E$ is a \ci, $\Ht F_{e}(E)= \Ht F_{e}(E_{\frp})$
    and $E_{\frp}$ is a \ci,
    for every prime $\frp \in \sR$ (by \refp{pdf}). Hence,
    $G_{\frp}/E_{\frp}$ is \CM{} (by the previous result).
    Therefore
    $\dpt G_{\frp}/E_{\frp}= \Ht \frp - \Ht F_{e}(E_{\frp})
    = \Ht \frp - \Ht F_{e}(E).$
    Moreover, by b), $\Ass R/F_{0}(G/E)$ $= \Ass R/F_{e}(E)$, and the equalities follow.
\end{proof}

Since any minimal reduction of an \eq{} module is a \ci, we may
assert the following.

\begin{corollary}\label{eci}
    Let $R$ be a \cm, $\dim R = d > 0$
    and $E$ a \eq{} module. Suppose that $E \subsetneq G \simeq R^e$.
    Then, for every minimal reduction $U$ of $E$,
       $$\Ass E/U \sube \Min R/F_{e}(E)= \{ \frp \in V(F_e(E)) \mid \Ht \frp = \ell(E) -e+1 \}.$$
\end{corollary}

\begin{proof}
    We may assume that $R$ has infinite residue field.

    Let $U$ be a minimal reduction of $E$. Since $U$ is \ci,
    $$ \Ass E/U \sube \Ass G/U = \Min R/F_{e}(U) = \Min
    R/F_{e}(E).$$
    (by \refp{eci1}). Moreover,
    \begin{align*}
     \frp \in \Min R/F_{e}(E) & \iff \frp \in \Ass G/U \iff \dpt G_{\frp}/U_{\frp}=0 \\
     & \iff \dim G_{\frp}/U_{\frp}=0 \iff \Ht \frp = \Ht
     F_e(U_{\frp}).
     \end{align*}
     (by \refp{dime}). Therefore,
     $$ \lE -e+1 = \Ht F_e(E) \leq \Ht \frp \leq \mu(U_{\frp}) -e+1 \leq \mu(U) -e+1= \lE -e+1.$$
     It follows that $\Min R/F_{e}(E)$
     is the set of all prime ideals $\frp \in V(F_e(E))$ such that $\Ht \frp = \ell(E) -e+1$.
\end{proof}

The following result extends to ideal modules a known criterion in
the ideal case by of D. Eisenbud, M. Hermann and W. Vogel (see
\cite[Theorem p. 179]{ehv}).

\begin{theorem}\label{cieq3}
    Let $R$ be a \cm, $E$ a non-free \fgtfrmo having rank $e >0$.
    Suppose that $E$ is \gci. Then $E$ is \ci{} \sse $E$ is \eq.
\end{theorem}

\begin{proof}
    Since $E$ is \ci{} [resp. \eq] \sse $E\prd = E \oxr R\prd$ is
    \ci{} [resp. \eq] for any Nagata extension $R\prd$ of $R$, we may
    assume that $R$ has infinite residue field.

    It is enough to prove that $E$ \eq{} implies that $E$ is \ci.
    So assume that $E$ is \eq{} and let $U$ be a minimal reduction
    of $E$. Hence $U$ is a \ci. Suppose that $U \subn E$.
    Let $\frp \in \Min R/F_{e}(E)$.
    Hence $U_{\frp}$ is a reduction of $E_{\frp}$ and,
    since $E$ is \gci, $E_{\frp}=U_{\frp}$. In particular
    $E_{\frp}=U_{\frp}$ for all $\frp \in \Ass E/U$
    (by \refc{eci}) - a contradiction. Therefore $E=U$ (and $\Ass
    E/U = \emptyset$).
\end{proof}

\begin{corollary}\label{cieq31}
     Let $R$ be a \cm, $\dim R=d>0$ and $E \subn G \simeq R^{e}$ an ideal module.
     Assume that $\Ht \frp = \lE -e +1$ for every $\frp \in \Min R/F_{e}(E)$ and
     that $E_{\frp}$ is a \ci{} for every prime $\frp \in \Min R/F_{e}(E)$. Then $E$ is a \ci.
\end{corollary}

\begin{proof}
    Since $\Ht F_{e}(E) = \Ht \frp$ for some $\frp \in \Min R/F_{e}(E)$ then $\lE=\Ht F_{e}(E) +e-1$ and $E$ is \eq.
Moreover, for any $\frp \in \Min R/F_{e}(E)$
$$ \ell(E_{\frp})= \mu(E_{\frp}) = \Ht F_{e}(E_{\frp})+e-1 \geq \Ht F_{e}(E) +e -1 = \lE \geq \ell(E_{\frp}),$$
and so $\Ht F_{e}(E)=\Ht F_{e}(E_{\frp})$, proving that $E$ is \gci. Therefore, by \reft{cieq3}, $E$ is a \ci.
\end{proof}

In \cite[Th\'eor\`eme 2]{mi}), A. Micali proved that $\Rm$ is
regular local \sse $\cS(\fm)$ is a domain. This result was an
important motivation to study the linear type property. We now
prove a criterion for an \eq{} module to be a \ci{} that extends
the above result of Micali.

\begin{proposition}\label{seq6}
    Let $R$ be a \nol{} and let $E$ be an ideal module. Then
    \begin{enumerate}
    \item[a)] $E$ is a \ci{} \sse $E$ is \eq{} and of linear type.
    \item[b)] If $\SE$ is a domain then $E$ is a \ci{} \sse $E$ is \eq.
    \end{enumerate}
\end{proposition}

\begin{proof}
We may assume that $k=R/\fm$ is infinite because any Nagata
extension $R\prd$ of $R$ has infinite residue field,
$\cS(E\prd)\simeq \cS(E) \oxr R\prd$ and $\CR(E\prd)\simeq \CR(E) \oxr R\prd$.
Suppose that $E$ is \eq. Then every minimal reduction of $E$ is a \ci{} (by \refp{pdf}). Since $E$ has no proper reductions (by \refc{as01}) then $E$ is a \ci, and a) is proved. Now if $\SE$ is a domain then $E$ is of linear type and b) follows by a).
\end{proof}



We recall that $E$ satisfies $\widetilde{G_{s}}$ if $\mu(E_{\frp})
\leq \dpt R_{\frp} +e -1$ for every $\frp \in \sR$ with $1 \leq
\dpt R_{\frp} \leq s-1$. Equivalently $E$ satisfies
$\widetilde{G_{s}}$ if $\grd F_{i}(E) \geq i-e+2$ for $e \leq i
\leq e+s-2$.
By \cite[Proposition 5]{av}, the symmetric algebra of an ideal
module over a domain with $\pjd E =1$ and satisfying
$\widetilde{G_{\mu(E)-e+1}}$ is a domain.
Hence we get the following consequence of \refp{seq6}.

\begin{corollary}\label{seq62}
    Let $R$ be a Noetherian  local domain and $E$ an ideal module
    having rank $e$ with $\pjd E=1$ and satisfying $\widetilde{G_{\mu(E)-e+1}}$.
    Then $E$ is a \ci{} \sse $E$ is \eq.
\end{corollary}

\smallbreak
\section{Examples of ideal modules with small reduction number}\label{sec:ex}

In this section we observe that finite direct sums of ideals of
grade $\geq 2$ are ideal modules, and give examples of \eq{} and
\gci{} modules with small reduction number.

\begin{proposition}\label{ds21}
    Let $R$ be a \nor{} and $E = \ses{E}{n}$ with $E_{i}$ \fg{} \tf \rmo s
    having rank $e_{i} >0$, $1 \leq i \leq n$, $n \geq 2$. Then $E$ is an ideal
    module \sse each summand $E_{i}$ is an ideal module.
\end{proposition}

\begin{proof}
    Suppose that $E_{i} \sube G_{i} \simeq R^{e_{i}}$
    and write $G= \ses{G}{n}$. Then $G$ is a free
    \rmo{} of rank $e= \sum_{i=1}^n e_{i} >0$ and $E=\ses{E}{n}
    \sube \ses{G}{n}=G$. Since $G/E=(\ses{G}{n})/(\ses{E}{n}) \simeq
    G_{1}/E_{1} \+ \cdots \+ G_{n}/E_{n}$ then
    $$ \supp G/E = \supp (G_{1}/E_{1} \+ \cdots \+ G_{n}/E_{n}) =
    \supp G_{1}/E_{1} \cup \cdots \cup \supp G_{n}/E_{n}.$$
    Therefore
    $$ \grd G/E =
    \min_{1 \leq i \leq n} \{ \grd G_{i}/E_{i} \} \geq 2 \iff \grd G_{i}/E_{i} \geq
     2, \; 1 \leq i \leq n, $$
    proving the equivalence.
\end{proof}

We observe that a direct sum of ideals cannot be a \ci{} module.

\begin{lemma}\label{ds1}
    Let $R$ be a \nor{} and $E = \ses{E}{n}$ with $E_{i}$
    \fgr s having positive rank $e_{i}$, $1 \leq i \leq n$, $n \geq 2$.
    Then
    $F_{e}(E) = F_{e_{1}}(E_{1}) \cdots F_{e_{n}}(E_{n})$, $e=
    \rank E$. In particular $\dis{\grd F_{e}(E) = \min_{1 \leq i \leq n} \{ \grd
    F_{e_{i}}(E_{i}) \}}$ and $\dis{\Ht F_{e}(E) = \min_{1 \leq i \leq n} \{ \Ht
    F_{e_{i}}(E_{i}) \}}$.
\end{lemma}

\begin{proof}
Since $E_{i}$ has rank $e_{i}>0$, then $F_{k}(E_{i}) = (0)$ for $k
< e_{i}$. Now
    $$ F_{e}(E) = F_{e}(\ses{E}{n}) = \sum_{j_{1}+ \cdots
    +j_{n}=e} F_{j_{1}}(E_{1}) \cdots F_{j_{n}}(E_{n}) =
    F_{e_{1}}(E_{1}) \cdots F_{e_{n}}(E_{n}).$$
The other assertions follow.
\end{proof}

\begin{proposition}\label{dsi2}
    Let $\Rmk$ be a \nol{} and $E = \ses{I}{e}$ with $I_{i}$
    $R$-ideals satisfying $\grd I_{i} \geq 2$ for all
    $1 \leq i \leq e$, $e \geq 2$.
    Then $E$ is not a \ci.
\end{proposition}

\begin{proof}
    By \refp{ds21}, $E$ is an ideal module.
    Hence $\Ht F_{e}(E) \geq \grd$ $F_{e}(E) \geq 2$.
    Suppose that $\Ht F_{e}(E) = h \geq 2$. Whence $\Ht I_{i} \geq h$
    for all $i$ (by \refl{ds1}), and so $\mu(I_{i}) \geq h$ for all $i$. It follows
    that
    $$ \mu(E) -e+1 = \sum_{i=1}^e \mu(I_{i}) -e + 1 \geq e \, h-e +
    1  \geq 2h-1 > h = \Ht F_{e}(E) $$
    and $E$ is not \ci.
\end{proof}


Suppose that $E = I \+ \cdots \+ I = I^{\+ e} \sube R^e$, $e \geq
2$, with $I$ an $R$-ideal. For any $k \in \N_{0}$,
$$ \CR(I^{\+ e})_{k} = (I^{k})^{\+ \binom{k+e-1}{k}} \simeq
\CR(\bI_{e})_{k},$$ where $\CR(\bI_{e})$ abbreviates the
multi-Rees algebra $\CR(I,\ldots,I)= R[It_1, \ldots, It_e]$.
Suppose that $r(I)\leq 1$ and let $J$ be a minimal reduction of
$I$ with $r_{J}(I)\leq 1$. Write $V= J \+ \cdots \+ J = J^{\+ e}
\sube E$. Then $I^2=J I$ and we have
$$E^{2} = (I^{2})^{\+ \binom{e+1}{2}} = (JI)^{\+ \frac{(e+1)e}{2}} =V \cdot E.$$
Therefore $V$ is a reduction of $E$ with $r_{V}(E) \leq 1$.
However, in general, $V$ is not a minimal reduction of $E$. In the
case where $E$ is \eq{} with $\Ht I=2$ we are able to construct a
minimal reduction $U$ of $E$ with $r_{U}(E)\leq 1$.

\begin{lemma}\label{pe0}
    Let $\Rm$ be a \nol{} with infinite residue field. Then, for each
    $n \geq 2$ there exist $\sek{\alp}{n} \in R$ such that
    $\alp_{i} - \alp_{j}$ is a unit for all $1 \leq i < j \leq n$.
\end{lemma}

\begin{proof}
    Let $n \geq 2$.
    Since $k=R/\fm$ is infinite, there exist $\sek{\alp}{n}\in R$
    such that
    $\alp_{i} + \fm \neq \alp_{j}+ \fm$ for all $1 \leq i < j \leq n$.
    It follows that $\alp_{i}-\alp_{j} \in R \setminus \fm =R^{*}$
    for all $1 \leq i < j \leq n$, proving the assertion.
\end{proof}

\begin{proposition}\label{pe1}
    Let $\Rm$ be a \cm{} with infinite residue field and
    $\dim R = d \geq 2$. Let $I$ be an \eq{} ideal with $\Ht I=2$ and
    $r(I) \leq 1$.
    Write $E = I \+ \cdots \+ I = I^{\+ e}$, $e \geq 2$. Then:
    \begin{enumerate}
        \item  $r(E) = 1$, $\lE=e+1$.

        \item  $E$ is \eq.
    \end{enumerate}
\end{proposition}

\begin{proof}
    Let $J$ be a minimal reduction of $I$ with $r_{J}(I)\leq 1$.
    Then $I^2=J I$.
    On the other hand, since $\lI=\Ht I =2$ then $J =
    \langle a,b \rangle$ for some $a, b \in I$.
    By the previous lemma there are
    $\sek{\alp}{e} \in R$ such that $\alp_{i}-\alp_{j} \in R^{*}$.
    Consider the family of elements
    $$ a_{1}=a, \; b_{1}=b, \; a_{i}=\alp_{i}a +b, \; b_{i}=a
    \;\;\;(i=2,\ldots,e).$$
    We have, for each $i,j$
    $$ a = (\alp_{i}-\alp_{j})^{-1}(\alp_{i} a +b)-
    (\alp_{i}-\alp_{j})^{-1}(\alp_{j} a +b) \; ,$$
    $$ b= - \alp_{j}(\alp_{i}-\alp_{j})^{-1}(\alp_{i} a +b) +
    (1+ \alp_{j} (\alp_{i}-\alp_{j})^{-1})(\alp_{j} a +b). $$
    Therefore
    $$ J = \langle a_{i}, b_{i} \rangle=
    \langle a_{i}, a_{j} \rangle \;\;\; (1 \leq i < j \leq e).$$
    Thus,
    $$ I^2=J I = a_{i}I + b_{i}I= a_{i}I + a_{j}I$$
    for all $1 \leq i < j \leq e$.
%
    Consider the elements $x_{i}=a_{i}\epsilon_{i}$, $1 \leq i \leq e$, and
    $y = \sum_{i=1}^e b_{i}\epsilon_{i}$, where $\seq{\epsilon}{e}$
    denotes the canonical basis of $R^e$, and consider
    the $R$-submodule $U$ of $E$ generated by $\sek{x}{e}, y$.
    We regard $\RE$ as a subalgebra of $R[\sek{t}{e}]$. So putting $I_{i}=I$ for $1 \leq i \leq e$
    \begin{align*}
        U \cdot E &= (Rx_{1} + \cdots + Rx_{e} + Ry) \cdot
        (\ses{I}{e}) \\
        & = (Ra_{1}t_{1} + \cdots + Ra_{e}t_{e} +
        R(b_{1}t_{1} + \cdots + b_{e}t_{e}))\cdot(I_{1}t_{1}+ \cdots +
        I_{e}t_{e}) \\
        & = \sum_{i=1}^e a_{i}I_{i}t_{i}^2 + \lpar \sum_{i=1}^e b_{i}
        t_{i} \rpar I_{1}t_{1} + \cdots +\lpar \sum_{i=1}^e b_{i}t_{i}
        \rpar I_{e}t_{e} + \sum_{1 \leq i < j \leq e}
        I_{i}I_{j}t_{i}t_{j},
    \end{align*}
    and
    \begin{align*}
        E^{2} & = I_{1}^2 \+ \cdots \+  I_{e}^2 \+ \bop_{1 \leq i<j
        \leq e} I_{i} I_{j} = \sum_{i=1}^e I_{i}^2 t_{i}^2 +
        \sum_{1 \leq i < j \leq e} I_{i}I_{j}t_{i}t_{j}\\
        & = \sum_{i=1}^e a_{i}I_{i}t_{i}^2 +
        \sum_{i=1}^e b_{i}I_{i}t_{i}^2 +
        \sum_{1 \leq i < j \leq e} I_{i}I_{j}t_{i}t_{j}.
    \end{align*}
    Since
    $$ \sum_{i=1}^e b_{i}I_{i}t_{i}^2 \sube \lpar \sum_{i=1}^e b_{i}
    t_{i} \rpar I_{1}t_{1} + \cdots +\lpar \sum_{i=1}^e b_{i}t_{i}
    \rpar I_{e}t_{e} + \sum_{1 \leq i < j \leq e}
    I_{i}I_{j}t_{i}t_{j}$$
    then $E^{2} \sube U \cdot E \sube E^{2}$.
    Therefore $U$ is a reduction of $E$ with $r_{U}(E) \leq 1$.
    Moreover, $\mu(U) \leq
    e+1 \leq \lE \leq \mu(U)$, that is $U$ is a minimal
    reduction of $E$. Thus $r(E) \leq r_{U}(E) \leq 1$ and
    $\lE=\mu(U)=e+1$.

    On the other hand, $E$ is an ideal module. Moreover, $\Ht F_{e}(E) = \Ht I=2$
    and $\lE=e+1$. Hence $\ad(E)= \lE -e+1 - \Ht F_{e}(E) = 0$, and so $E$ is \eq.
    Therefore, by \refp{pdf} and by \refp{dsi2}, $r(E)=1$.
\end{proof}

In particular, if $I$ is a \ci{} ideal we obtain examples of \eq{}
modules with reduction number equal to $1$.

\begin{corollary}\label{pe2}
    Let $R$ be a \cm{} with infinite residue field and
    $\dim R = d \geq 2$. Let $I$ be a \ci{} ideal with $\Ht I=2$.
    Write $E = I \+ \cdots \+ I = I^{\+ e}$, $e \geq 2$. Then
    $E$ is \eq{} with $r(E) = 1$ and $\lE=e+1$.
\end{corollary}

\begin{corollary}\label{pe3}
    Let $\Rm$ be a regular local ring with infinite residue field and $\dim R = d = 2$.
    Let $E = \fm \+ \cdots \+ \fm = {\fm}^{\+ e}$ with $e \geq 1$.
    Then $E$ is \eq{} with $r(E) = 1$ and $\lE=e+1$.
\end{corollary}



Next, we give examples of \gci{} modules which are a direct sum of prime ideals.

\begin{proposition}\label{pp1}
    Let $R$ be a \cm{} with infinite residue field and $\dim R=d\geq 3$.
    Let $\sek{\frp}{e}$ be pairwise distinct prime ideals which are
    perfect of grade $2$. Write $E= \ses{\frp}{e}$, $e \geq 2$.
    Then:
    \begin{enumerate}
        \item  $\mu(E_{\frp}) \leq \Ht \frp +e-1$ for every prime $\frp$ with $1 \leq \Ht \frp \leq 2$.

        \item  $E$ is \gci.

        \item  $E$ is not \eq.



        \item  $\lE \geq e+2$, $\ad(E)\geq 1$ with equalities if $d=3$.

        \item If $d=3$, $e=2$ and $\frp_{1}$, $\frp_{2}$ are
        \ci{} then $r(E)=0$.
    \end{enumerate}
\end{proposition}

\begin{proof}
    a)
    Let $\frp \in \sR$.
    If $\Ht \frp =1$ then $\frp \neq \frp_{i}$ for all $i$. Hence
    ${\frp_{i}}_{\frp}=\frp_{i}R_{\frp}= R_{\frp}$ for all $i$.
    Thus
    $$E_{\frp} = {\frp_{1}}_{\frp} \+ \cdots \+ {\frp_{e}}_{\frp}
    \simeq R_{\frp}^e,$$
    and so $\mu(E_{\frp})=e$. Now, suppose that
    $\Ht \frp = 2$. Then either $\frp \neq \frp_{i}$ for all $i$,
    so $\mu(E_{\frp})=e \leq e+1$, or $\frp = \frp_{j}$ for some $j$
    and $\frp \neq \frp_{i}$ for all $i \neq j$.
    In this case, ${\frp_{j}}_{\frp}=\frp R_{\frp}$.
    Moreover, since $\frp_{j}$ is perfect of grade $2$, then $\pjd
    R_{\frp}/\frp R_{\frp} < \infty$. By \cite[Theorem~2.2.7]{bh},
    $R_{\frp}$ is a regular local ring and so $\frp R_{\frp}$ is a
    \ci{}. Hence $\mu(\frp R_{\frp}) =
    \Ht \frp R_{\frp} =2$. Therefore
    $$ E_{\frp} ={\frp_{1}}_{\frp} \+ \cdots \+ {\frp_{j}}_{\frp}
    \+ \cdots \+ {\frp_{e}}_{\frp} \simeq R_{\frp} \+ \cdots \+ \frp
    R_{\frp} \+ \cdots \+ R_{\frp} \simeq R_{\frp}^{e-1} \+ \frp
    R_{\frp},$$
    so that $\mu(E_{\frp}) = e-1 + \mu(\frp R_{\frp}) = e+1$.

    b)  By \refc{ds21}, $E$ is an ideal module. Let $\frp \in \Min
    R/F_{e}(E)$. We have $\frp \in V(F_{e}(E))= V(\frp_{1} \cdots
    \frp_{e})$ and $\Ht \frp = \Ht F_{e}(E)=2 = \Ht \frp_{i}$, for all $i$.
    Hence $\frp = \frp_{j}$ for some $j$ and $\frp \neq \frp_{i}$ for $i\neq j$.
    Therefore, as above,
    $$ \mu(E_{\frp}) =e+1 = \Ht F_{e}(E)+e-1,$$
    proving that $E$ is \gci.

    c)  follows by \refp{cieq3} and \refp{dsi2}.



    d)  Since $E$ is \gci{} then $\ell(E_{\frp})= \mu(E_{\frp}) = \Ht F_{e}(E)+e-1=e+1$
    for all $\frp \in \Min R/F_{e}(E)$. Hence $\lE \geq e+1$ (by \refc{as0}).
    If $\lE=e+1$ then $E$ is \eq.
    It follows that $\lE \geq e+2$. Hence $\ad(E)=\lE-e+1-\Ht F_{e}(E)\geq 1$.
    Moreover, if $d=3$ then $\lE \leq d+e-1 = e+2$.
    Hence $\lE=e+2$ and so $\ad(E)=1$.

    e) Since $d=3$ and $\frp_{1}$, $\frp_{2}$ are c.i.
    ideals, then $\mu(E)= \mu(\frp_{1})+\mu(\frp_{2}) = \Ht
    \frp_{1} + \Ht \frp_{2}=4= \lE$. Hence $E$ is a minimal reduction
    of itself, that is $r(E)=0$.
\end{proof}

We note that a direct sum of \eq{} modules is not always an \eq{}
module. In the situation below, $E$ is a direct sum of \ci{}
ideals but $E$ is not \eq{} (cf. \reft{cieq3} and \refp{pp1}).

\begin{example}\label{exs3}
    Let $R= k[[X_{1}, X_{2}, X_{3}]]$ with $k$ an infinite field and
    write $E= \langle X_{1}, X_{2} \rangle \+ \langle X_{1}, X_{3}
    \rangle$. Then $E$ is \gci{} with $\lE=4$, $\ad(E)=1$ and $r(E)=0$.
\end{example}

\begin{proof}
    $R$ is a regular local ring with maximal ideal $\fm = \langle
    X_{1}, X_{2}, X_{3} \rangle$, dimension $d=3$ and $\frp_{1}=
    \langle X_{1}, X_{2} \rangle$, $\frp_{2} = \langle X_{1}, X_{3}
    \rangle$ are two distinct prime ideals of $R$ with $\Ht \frp_{i} =
    2 = \mu(\frp_{i})$, $i=1,2$.
    The assertions then follow by \refp{pp1}.
\end{proof}

\smallbreak
\section{Open conditions}\label{sec:open}

\noindent Given a ring $R$, an \rmo{} $E$ and a property $P$ it is
very important to know if the subset $\{ \frp \in \sR \mid
E_{\frp} \text{ satisfies } P \} \subset \sR$ (the P locus) is
open. For instance, for a \fg{} \rmo{} $E$ over a \nor{} $R$, it
is well known that $U_{n}=\{ \frp \in \sR \mid \mu(E_{\frp}) \leq
n \}$ and $U_{F}=\{ \frp \in \sR \mid E_{\frp}\text{ is free over
} R_{\frp}\}$ are open subsets in $\sR$.
\smallbreak


\begin{lemma}\label{alpha}
    Let $R$ be a \nor, $0\neq I$ an ideal in $R$ and $\frp \in V(I)$.
    Then, there exists $\alpha \notin \frp$ such that $\Ht IR_{\alpha} = \Ht I_{\frp}$.
\end{lemma}

\begin{proof}

Let
$$ J = \biind{\bigcap}{\frq \in \Min V(I)}{\quad \; \frq \nsub \frp} \frq. $$
If $J=\emptyset$ choose any $\alpha \notin \frp$. If $J\neq
\emptyset$, then $0\neq J$ is not contained in $\frp$ and there
exists
    $\alpha \in J$ such that $\alpha \notin \frp$. Hence, for any prime ideal
    $\frq \in \Min V(I)$ such that $\alpha \notin \frq$, $\frq \sube \frp$.
    In both cases it is now clear that $\Ht IR_{\alpha} = \Ht I_{\frp}$.



\end{proof}

The openness of the complete intersection locus of an ideal is
well known. For ideal modules we have the following:

\begin{theorem}\label{openci}
Let $R$ be a \nor{} and $E \subn G \simeq R^{e}$ an ideal module.
Then $U_{ci}= \{ \frp \in \supp G/E \mid E_{\frp} \,
\textrm{is a complete intersection}\, \}$ is a (possibly empty)
open subset in $\supp G/E$.
\end{theorem}

\begin{proof}
     Let $\frp \in U_{ci}$ with $r=\mu (E_{\frp})$.
     Now, let $\alpha \notin \frp$ such that $\Ht F_e(E)R_{\alpha}=\Ht F_e(E_{\frp})$ (by Lemma
     \ref{alpha}). Then, $U_r \cap D(\alpha) = \{ \frq \in D(\alpha) \mid \mu(E_{\frq}) \leq r \,\}$ is an open,
     non-empty set containing $\frp$, such that for any
     $\frq \in U_r \cap D(\alpha) \cap \supp G/E$, $E_{\frq}$ is an ideal module and
     \begin{align*}
     \mu (E_{\frq}) & \leq r = \mu (E_{\frp}) = e - 1 + \Ht
     F_e(E_{\frp}) \\
     & = e -1 + \Ht F_e(E)R_{\alpha} \leq e - 1 + \Ht
     F_e(E_{\frq})\leq \mu (E_{\frq})\,.
     \end{align*}
     Hence, $$\mu (E_{\frq}) = e - 1 + \Ht F_e(E_{\frq})$$ and $\frq \in U_{ci}$.
\end{proof}

We note that $U_{ci}$ may be an empty set (cf. \refp{dsi2}).
\smallbreak

\begin{remark}\label{rpp}
    Let $R$ be a \nor{} and let $E \sube G \simeq R^e$,
    $e>0$, be an \rmo. Let $\frp \in \sR$. Then, for every n,
    $(E^{n})_{\frp} \simeq ({E_{\frp}})^{n}.$
    We simply write $E^{n}_{\frp}$ in any case.
\end{remark}

The openness of the equimultiple locus of an ideal is proven, for
instance, in \cite{hio}. Similarly, for ideal modules we get:

\begin{theorem}\label{openeq}
Let $R$ be a \nor{} and $E \subn G \simeq R^{e}$ an ideal module.
Then $U_{eq}= \{ \frp \in \supp G/E \mid E_{\frp} \, \textrm{is
equimultiple} \, \}$ is a non-empty open subset in $\supp G/E$.
\end{theorem}

\begin{proof}
     For any $\frp \in \supp G/E = V(F_{e}(E))$, $E_{\frp}$ is an ideal module and so
     $\Ht F_{e}(E) \leq \Ht F_{e}(E_{\frp}) \leq \ell(E_{\frp})-e+1 \leq \Ht \frp$.
     Now let $\frp \in V(F_{e}(E))$ minimal such that $\Ht \frp = \Ht F_{e}(E)$.
     Hence $\frp \in U_{eq}$ and so $U_{eq}$ is non-empty.

     Let $\frp \in U_{eq}$ with $s=\ell (E_{\frp})$ and $\sek{a}{s} \in E$ such that $\frac{a_1}{1}, \ldots, \frac{a_s}{1}$
     is a homogeneous system of parameters of $\mathcal F (E_{\frp})$, see \refr{hsop}. Hence, for some $r$,
     $E_{\frp}^{r}= a_{1}E_{\frp}^{r-r_1}+ \cdots + a_{s}E_{\frp}^{r-r_s}$.
     Therefore $\ann_{R}(E^{r}/a_{1}E^{r-r_1}+ \cdots + a_{s}E^{r-r_s}) \nsub \frp$.
     Thus, by \refl{alpha}, we may choose $\alpha \in R \setminus \frp$ such that
     $$ \alpha E^{r} \sube a_{1}E^{r-r_1} + \cdots + a_{s}E^{r-r_s} \text{ and }
     \Ht F_{e}(E)R_{\alpha} = s-e+1.$$
     Therefore, for any $\frq \in D(\alpha) \cap \supp G/E$ we have
     $$ \alpha E_{\frq}^{r}= E_{\frq}^{r} $$
and so $E_{\frq}^{r}= a_{1}E_{\frq}^{r-r_1} + \cdots +
a_{s}E_{\frq}^{r-r_s}$, showing that $\ell(E_{\frq}) \leq s$. Thus
we get
$$ s \geq \ell(E_{\frq}) \geq \Ht F_{e}(E_{\frq})+e-1 \geq \Ht F_{e}(E)R_{\alpha}+e-1 = s, $$
which proves that $D(\alpha) \cap \supp G/E \sube U_{eq}$. Therefore
$$ U_{eq} = \bigcup_{\alpha}D(\alpha) \cap \supp G/E,$$
and so $U_{eq}$ is an open subset in $\supp G/E$, as required.
\end{proof}

\smallbreak
\section{The Rees powers of an ideal module}\label{sec:rp}

\noindent In \cite{cz1} we defined the $n$-th Rees power $E^n$ of
a \fg{} \rmo{} $E \sub G \simeq R^e$ as the $n$-th graded piece of
$\RE$
$$ E^{n} := \RE_{n}\sub G^n \simeq R[t_1, \ldots, t_e]_n \simeq
    R^{\binom{n+e-1}{e-1}},$$
in order to prove the Burch's inequality for modules. \medbreak

\noindent Computing Rees powers of a module seems to be rather
complicated even in the easiest cases.

\begin{proposition}\label{p1}
    Let $\Rm$ be a \nol{} with $\dim R = d >0$.
    Suppose that $E = F \+ I$ where $F \simeq R^{e-1}$ and $I$ is an $R$-ideal.
    Then
    \begin{enumerate}
    \item $E^{n}\simeq \bop_{j=0}^{n}I^{j}R[\sek{t}{e-1}]_{n-j}$;
    \item $\dpt E^{n}= \min_{0 \leq j \leq n} \dpt I^{j}$;
    \item $\grd \fm \RE = \grd \fm \RI$.
    \end{enumerate}
\end{proposition}

\begin{proof}
    Since $F$ is a free module of rank $e-1$ then $\RE \simeq \RI[\sek{t}{e-1}]$ and a) follows.
    For b) note that $\dpt J = \dpt JG$ for any ideal $J$ and any free module $G$.
    Hence, by a),
    $$  \dpt E^{n}= \min_{0 \leq j \leq n} \dpt I^{j}R[\sek{t}{e-1}]_{n-j}= \min_{0 \leq j \leq n} \dpt I^{j}.$$
    Finally, using \cite[Lemma~5.1]{cz1}
    $$ \grd \fm \RE = \inf_{n \geq 0} \dpt E^{n} = \inf_{n \geq 0} \dpt I^{n} = \grd \fm \RI,$$
    proving c).
\end{proof}

In the case where $E= I_1 \+ \cdots \+ I_e$ with $\grd I_i >0$,
$i= 1, \ldots,e$, then $ \RE = \CR(I_1 \+ \cdots \+ I_e)$
and so
$$E^n\simeq \bop_{k_1+\cdots+k_e=n}I^{k_1}\cdots I^{k_e}.$$
In particular, in the case where $I_1 = \cdots = I_e=I$
abbreviating $I \+ \cdots \+ I=I^{\+ e}$ we get
$$ E^n = I^n \+ \cdots \+ I^n = (I^n)^{\+ \binom{n+e-1}{n}}.$$
In this case, $\dpt E^n = \dpt I^n$, for every $n \geq 1$.
\medbreak

Now we establish some basic properties about the quotients $G^n
/E^n$ for general \fg{} \tf \rmo s. In fact, we prove that $G^n
/E^n$ has the same support, the same dimension and the same grade
as $G/E$, and we apply this to ideal modules.
\medbreak



\begin{proposition}\label{rp1}
    Let $R$ be a \nor{} and let $E \subn G \simeq R^e$,
    $e>0$, be an \rmo.
    Then, for every $n \geq 1$,
    \begin{enumerate}
    \item  $\supp G^n /E^{n} = \supp G/E$;

    \item  $\dim G^n /E^{n} = \dim G/E$;

    \item  $\grd G^n /E^{n} = \grd G/E$;

    \item  $\Min G^n /E^{n} = \Min G/E$;

    \item  $G^n=E^{n}  \iff G=E$.

    \item  $1 \leq \dpt E^{n} \leq \dpt G^n/E^{n}+1 \leq d-1$ if $R$ is \CM.
    \end{enumerate}
\end{proposition}

\begin{proof}
Let $n \geq 1$.

a)  The inclusion ``$\sube$'' is clear.
    On the other hand, suppose that $\frp \in \sR \setminus
    \supp G^n /E^{n}$ then
    $$  E_{\frp}^{n} \sube E_{\frp}G_{\frp}^{n-1} \sube
    G_{\frp}^n = E^{n}_{\frp},$$
    and so $E_{\frp}G_{\frp}^{n-1}= G_{\frp}^n$, that is $E_{\frp}$
    is a reduction of $G_{\frp}$. But $G_{\frp}$ is a
    free $R_{\frp}$-module, hence $G_{\frp}$ has no proper
    reductions.
    Thus $E_{\frp}=G_{\frp}$ and so $\frp \not \in \supp G/E$.

b) -- e) are direct from a).

f) Since $R$ is \CM, $\dpt G^n = \dpt R=\dim R$. Now the
   ine\-qua\-li\-ty follows by the depth Lemma applied to the exact
   sequence $0 \ra E^{n} \ra G^n \ra G^n/E^{n} \ra 0$.
\end{proof}

In the following $E$ will be an ideal module. In this case we
deduce some consequences of the result above. In fact, the
equality $\supp G/E = \supp G^n/E^n$ implies that $E^n$ is an
ideal module, and so any result proved for $G/E$ holds for
$G^n/E^n$.

\begin{corollary}\label{rp11}
    Let $R$ be a \nor, $E$ an ideal module having
    rank $e >0$. Then, for every $n \geq 1$, $E^n$ is an ideal module having rank $e_{n}=\binom{n+e-1}{e-1}$.
    Moreover, $V(F_{e_{n}}(E^{n}))=V(F_{e}(E))$.
\end{corollary}

\begin{proof}
    The assertion is clear if $E$ is free. Now suppose that $E$ is
    not free and let $n \geq 1$.
    We first note that $E^n \neq 0$.
    Suppose that $E \subn G \simeq R^e$.
    If $E^n =0$ for some $n$, then
    $$ \supp G/E = \supp G^n/E^n = \supp G^n = \sR. $$
    But $E_{\frp} \simeq R_{\frp}^e \simeq G_{\frp}$ for every $\frp
    \in \Ass E$ and $\grd G_{\frp}/E_{\frp} \geq 2$. Hence $E_{\frp}
    = G_{\frp}$ (cf. \refl{grd>2}) - a contradiction. Therefore,
    $E^n \neq 0$.
    Moreover, by the previous result, $E^n \subn G^n \simeq
    R^{\binom{n+e-1}{e-1}}$
    with $\grd G^n/E^n= \grd G/E \geq 2$, and so $E^{n}$ is an
    ideal module having rank $e_{n}$.
\end{proof}

\begin{corollary}\label{rp2}
    Let $R$ be a \cm{} with $\dim R = d \geq 2$ and let
    $E \subn G \simeq R^e$, be
    an ideal module having rank $e > 0$.
    If $G/E$ is \CM{} then $\dpt E^n \leq \dpt E$.
\end{corollary}

\begin{proof}
    Since $G/E$ is \CM,
    \begin{align*}
        \dpt E^n & = \dpt G^n/E^n +1 \leq \dim G^n/E^n +1 = \dim G/E
        +1 \\
         & = \dpt G/E +1 = \dpt E,
    \end{align*}
    as asserted.
\end{proof}

In the case where $\dim R = 2$, we have $\dpt E^n = 1$ and
$G^n/E^n$ is \CM{} of dimension $0$, for all $n \geq 1$, as in
\refc{dime2}.

\begin{corollary}\label{rp20}
    Let $R$ be a \cm, $\dim R = 2$ and let $E \subn G \simeq R^e$ with
    $\grd G/E \geq 2$, be
    an ideal module over $R$ having rank $e > 0$.
    Then, for every $n \geq 1$,
    \begin{enumerate}
    \item  $\dpt E^n = 1$;

    \item  $\dim G^n/E^n = \dpt G^n/E^n = 0$.
    \end{enumerate}
\end{corollary}

It is known that the analytic spread $\lI$ of an ideal $I$ over a \nol{} satisfies the inequality
$$ \lI \leq \dim R - \inf_{n \geq 1} \dpt R/I^{n}$$
called the Burch's inequality. In \cite{cz1} we proved that
$$ \lE \leq \dim R +e-1 - \inf_{n \geq 1} \dpt G^{n}/E^{n},$$
for a module $E \sub G \simeq R^{e}$. To do this we first proved
that $\dpt G^{n}/E^{n}$ takes a constant value for large $n$. For
equimultiple modules we are able to give an easy proof of this
inequality.

\begin{proposition}\label{rp3}
    Let $R$ be a \nol{} with $\dim R = d \geq 2$ and let
    $E \subn G \simeq R^e$ be
    an \eq{} \rmo{} having rank $e \geq 2$. Then
    $$ \ell(E) \leq d +e-1 - \inf_{n \geq 1} \dpt G^n /E^n.$$
\end{proposition}

\begin{proof}
    We have, for every $n \geq 1$,
    $$ \dpt G^n /E^n \leq \dim G^n /E^n = \dim G/E
    \leq d - \Ht F_{e}(E).$$
    Therefore, since $E$ is \eq,
    $$ \inf_{n \geq 1} \dpt G^n /E^n  \leq d - \Ht F_{e}(E) = d
      - \ell(E) +e -1, $$
    and the required inequality follows.
\end{proof}

In the case where the Rees algebra $\RE$ is \CM{} the Burch's
ine\-qua\-lity is an equality (cf. \cite[Corollary~5.3]{cz1}). In
this case we obtain the following characterization for \eq{}
modules.

\begin{proposition}\label{eqbeq}
    Let $R$ be a \cm, $\dim R =d >0$ and let $E \subn G \simeq R^e$
    be an ideal module having rank $e>0$ but not free.
    If $\RE$ is \CM{} then
    the following are all equivalent:
    \begin{enumerate}
    \item  $E$ is \eq;

    \item  $\dpt G^n /E^n=d -\Ht F_{e}(E)$ for all $n>0$;

    \item  $\dpt G^n /E^n=d -\Ht F_{e}(E)$ for infinitely many $n$.
    \end{enumerate}
\end{proposition}

\begin{proof}
    a) $\Rightarrow$ b). Since $E$ is \eq{} then
    $$ \Ht F_{e}(E) +e-1 = \lE = d+e-1- \inf_{n \geq 1} \dpt
    G^{n}/E^{n},$$
    and so $\dpt G^n /E^n \geq \inf_{n \geq 1} \dpt G^{n}/E^{n}
    = d -\Ht F_{e}(E) = \dim G/E$, for all $n>0$.

    b) $\Rightarrow$ c) is immediate.

    c) $\Rightarrow$ a) follows by \cite[Corollary~6.2]{cz1}.
\end{proof}

\begin{corollary}\label{eqbeq1}
    Let $R$ be a \cm, $\dim R =d >0$ and let $E \subn G \simeq R^e$
    be an ideal module having rank $e>0$ but not free.
    If $E$ is \ci{} then $G^{n}/E^{n}$ are \CM{} and $\Ass G^{n}/E^{n}= \Ass G/E= \Min R/F_{e}(E)$, for all $n \geq 1$.
\end{corollary}

\begin{proof}
    In this case $\RE$ is \CM{} and the assertions follow by
    \refps{eci1}, \ref{rp1} and \ref{eqbeq}.
\end{proof}



\end{document}